\documentclass[a4paper,11pt]{article}

\usepackage[utf8]{inputenc}
\usepackage[T1]{fontenc}

\usepackage{amssymb}
\usepackage{amsmath}
\usepackage{amsthm}
\usepackage{enumerate}
\usepackage{tikz}
\usepackage{url}
\usepackage{verbatim}
\usepackage{hyperref}

\usetikzlibrary{matrix,arrows}

\theoremstyle{plain}
\newtheorem{theo}{Theorem}

\newtheorem{prop}{Proposition}
\newtheorem{defi}{Definition}

\theoremstyle{remark}

\newtheorem{rema}{Remark}

\newcommand{\QH}{\mathrm{QH}}
\newcommand{\G}{\mathrm{G}}

\renewcommand{\H}{\mathrm{H}}
\newcommand{\HH}{\mathcal{H}}
\newcommand{\IG}{\mathrm{IG}}
\newcommand{\OG}{\mathrm{OG}}
\newcommand{\ra}{\rightarrow}
\newcommand{\C}{\mathbb{C}}
\newcommand{\F}{\mathrm{F}}
\renewcommand{\P}{\mathbb{P}}
\newcommand{\Z}{\mathbb{Z}}
\newcommand{\Q}{\mathbb{Q}}
\newcommand{\EE}{\mathcal{E}}
\renewcommand{\SS}{\mathcal{S}}
\newcommand{\OO}{\mathcal{O}}
\newcommand{\relmid}[1]{\mathrel{}\middle#1\mathrel{}}

\title{Quantum product and parabolic orbits in homogeneous spaces}
\author{Cl\'{e}lia Pech - Institut Fourier \\ \url{clelia.pech@ujf-grenoble.fr}}

\begin{document}

\maketitle

\begin{abstract}
 Chaput, Manivel and Perrin proved in \cite{CMP4} a formula describing the quantum product by Schubert classes associated to cominuscule weights in a rational projective homogeneous space $X$. In the case where $X$ has Picard rank one, we link this formula to the stratification of $X$ by $P$-orbits, where $P$ is the parabolic subgroup associated to the cominuscule weight. We deduce a decomposition of the \emph{Hasse diagram} of $X$, \emph{i.e} the diagram describing the cup-product with the hyperplane class.
\end{abstract}

\section{Introduction}

Let $G$ be a semisimple algebraic group over $\C$, $B$ be a Borel subgroup and $T \subset B$ a maximal torus. We denote by $\Phi$ the set of roots of $G$ with respect to $T$, $\Phi^+$ the subset of positive roots with respect to $B$, $\Delta = \left\{ \alpha_1, \dots, \alpha_n \right\}$ the subset of simple roots and $W$ the Weyl group of $G$. A fundamental weight $\omega$ is said to be \emph{minuscule} if $|\langle \alpha^\vee,\omega \rangle| \leq 1$ for all $\alpha\in\Phi$, where $\alpha^\vee$ is the coroot of $\alpha$. It is said to be \emph{cominuscule} if it is minuscule for the dual root system. Fundamental weights will be denoted $\omega_1, \dots, \omega_n$, with the same order as in the notation of \cite{bbk}.

Let $Q \supset B$ be a parabolic subgroup of $G$ and denote by $X$ the homogeneous space $G/Q$. In \cite{CMP4}, Chaput, Manivel and Perrin proved a formula describing the quantum product in $X$ by special Schubert classes associated to cominuscule weights. These classes correspond to the elements of the image of Seidel's representation $\pi_1(G^\mathrm{ad}) \ra \QH^*(G/Q)^\times_\mathrm{loc}$ \cite{seidel}, where $G^\mathrm{ad}=G/Z(G)$ and $\QH^*(G/Q)^\times_\mathrm{loc}$ is the group of invertible elements in the small quantum cohomology ring $\QH^*(G/Q)$ localized in the quantum  parameters. Before stating this result, we introduce some notation for the quantum cohomology of $X$.

The quantum cohomology ring $\QH^*(X)$ of a homogeneous variety $X=G/Q$ is a deformation of its cohomology ring. Consider the parameter ring
\[
 \Lambda = \left\{ \sum_\beta a_\beta q^\beta \relmid{|} \beta\in\H_2^+(X,\Z), a_\beta \in\Z \right\},
\]
where the sums are finite, $\H_2^+(X,\Z)$ denotes the set of effective cycles in $H_2(X,\Z)$ and the $q^\beta$ are formal parameters such that $q^\beta q^{\beta'} = q^{\beta + \beta'}$.
As a $\Z$-module, the quantum cohomology ring $\QH^*(X)$ is isomorphic to $\H^*(X,\Z) \otimes_{\Z} \Lambda$.
Moreover, it admits a ring structure defined by the \emph{quantum product} $\star$, which is a deformation of the cup-product. A precise definition for the quantum product can be found in \cite{FP}. The group $H_2(X,\Z)$ contains $\Phi^\vee / \Phi_Q^\vee$, where $\Phi^\vee$ denotes the coroot lattice of $G$ and $\Phi^\vee_Q$ the coroot lattice of $Q$, hence positive coroots can be seen as effective classes $\beta\in\H_2^+(X,\Z)$.

Now let $\mathcal{I}$ be the set of vertices of the Dynkin diagram of $G$ corresponding to cominuscule weights. If $i\in \mathcal{I}$, let $v_i$ be the shortest element of the Weyl group $W$ such that $v_i\omega_i^\vee=w_0\omega_i^\vee$, where $\omega_i^\vee$ is the fundamental coweight associated to $i$ and $w_0$ is the longest element of $W$. Then the quantum product in $X$ by the Schubert class $\sigma_{v_i}$ Poincar\'{e} dual to the Schubert cycle $[X_{v_i w_0}]$ is given by the following formula :
\begin{theo}[{\cite[Thm.1]{CMP4}}]
\label{th:para.formula}
 For all $w\in W$ and for all $i\in  \mathcal{I}$, we have :
\[
 \sigma_{v_i}\star\sigma_{w}=q^{\eta_Q(\omega_i^\vee-w^{-1}(\omega_i^\vee))}\sigma_{v_i w},
\]
where $\eta_Q : \Phi^\vee\ra \Phi^\vee/\Phi^\vee_Q$ is the natural surjection.
\end{theo}

The aim of this paper is to relate the above theorem to a stratification of $X=G/Q$ by $P_i$-orbits when $Q$ is a maximal parabolic and $P_i$ is the maximal parabolic associated to the weight $\omega_i$.
In Section \ref{sec:para.def}, we recall some well-known facts about parabolic orbits and we describe the parabolic orbits associated to cominuscule weights in the classical Grassmannians. Then in Section \ref{sec:link.para}, we explain the link between Thm. \ref{th:para.formula} and the stratification by parabolic orbits in $X$. We deduce in Section \ref{sec:Hasse} a decomposition of the Hasse diagram of the classical Grassmannians. 

\subsection*{Acknowledgements}

This paper reports on work done during my thesis. I would like to thank my advisor, Laurent Manivel, for his support and for valuable discussions. 
I am also indebted to Nicolas Perrin for many useful remarks and comments.

\section{Parabolic orbits}
\label{sec:para.def}

In \ref{subsec:para.def} we recall some classical facts about parabolic orbits in (generalized) flag varieties, and in \ref{subsec:para.class}, we give a more explicit description of parabolic orbits associated to cominuscule weights in the classical Grassmannians.

\subsection{Parabolic orbits in generalized flag varieties}
\label{subsec:para.def}

A \emph{generalized flag variety} is a variety of the form $X=G/Q$, where $G$ is a semisimple algebraic group, $B$ a Borel subgroup and $Q\supset B$ a  parabolic subgroup. Now consider a second parabolic subgroup $P\supset B$. We call \emph{$P$-orbits} or \emph{parabolic orbits} the orbits of $X$ under the action of $P$ by left multiplication. Here are some elementary properties of parabolic orbits, which can be found in \cite[Sec. 2.1]{perrin} :

\begin{prop}
\label{prop:par.base}
 \begin{enumerate}
  \item Every $P$-orbit can be written as $P w Q/Q$ with $w\in W$. 
  \item The $P$-orbits are smooth and locally closed, indexed by the double cosets $W_P \backslash W / W_Q$, where $W_P$ and $W_Q$ denote the Weyl groups associated to $P$ and $Q$. Moreover, they define a stratification of $X$ :
    \[
     X = \bigsqcup_{W_P w W_Q \in W_P \backslash W / W_Q} P w Q/Q.
    \]
  \item\label{it:par.base} The $P$-orbits are $B$-stable, hence they are a union of Schubert cells :
    \[
     P w Q/Q = \bigcup_{(w_P,w_Q)\in W_P\times W_Q} B w_P w w_Q Q / Q.
    \]
 \end{enumerate}
\end{prop}

We denote by $W^Q$ the set of minimal length representatives of cosets in $W/W_Q$, which inherits the Bruhat order of $W$. Let us describe the double cosets indexing parabolic orbits :
\begin{prop}
\label{prop:interval}
 Let $\EE=W_P w W_Q$ be a double coset in $W_P\backslash W/W_Q$. Then $\EE\cap W^Q$ contains unique minimal and maximal elements $w_{min}$ and $w_{max}$. Moreover, it is equal to the interval $\left[w_{min},w_{max}\right]$ for the Bruhat order in $W^Q$.
\end{prop}

\begin{proof}
 This statement is an exercise (without proof) in \cite[Chap. 4, § 1]{bbk}. Here we give a geometric proof. 

 Let $\OO$ be the $P$-orbit indexed by $\EE$. By Item \ref{it:par.base} of Prop. \ref{prop:par.base}, we have
 \[
  \OO = \bigcup_{w'\in\EE} C_{w'},
 \]
 where $C_{w'}=B w' Q/Q$ is the Schubert cell associated to $w'$. Hence the closure $\overline{\OO}$ of $\OO$ satisfies
 \[
  \overline{\OO} = \bigcup_{w'\in\EE} X_{w'},
 \]
 where $X_{w'} = \overline{C}_{w'}$ is the Schubert variety associated to $w'$. Moreover, $\overline{\OO}$ being $B$-stable, irreducible and closed, it is a Schubert variety $X_{w_{max}}$ with $w_{max}\in W^Q$. It follows that $w_{max}$ is in $\EE\cap W^Q$, and by definition $w'\leq w_{max}$ for each $w' \in \EE\cap W^Q$.

 Similarly, the double coset $\EE':=W_Q w^{-1} W_P$ admits a unique maximal element $\widetilde{w}\in W^Q$. Then $w_{min}:=\widetilde{w}$ is a unique minimal element in $\EE\cap W^Q$.

 Let us now prove that $\EE\cap W^Q$ is the interval $\left[w_{min},w_{max}\right]$. From the definition of $w_{min}$ and $w_{max}$, it is already clear that $\EE\cap W^Q$ is contained in the aforementioned interval. Moreover, the Schubert cell $C_{w_{min}}$ is contained in all Schubert cells $C_{w'}$ for $w'\in\EE\cap W^Q$. The boundary $\overline{\OO}\setminus\OO$ being closed and $B$-stable, it is a union of Schubert varieties $X_{w_i}$, and we may write :
 \[
  \overline{\OO}\setminus\OO = \bigsqcup_{i=1}^r X_{w_i}
 \]
 for some $w_i\in W^Q$. Let $C_{w'}$ be a Schubert cell in $X_{w_{max}}=\overline{\OO}$, where $w' \in W^Q$. It means that $w' \in\left[1,w_{max}\right]$. The cell $C_{w'}$ lies in $\OO$ if and only if $C_{w'} \not\subset X_{w_i}$, \emph{i.e} $w'\not\in\left[1,w_i\right]$ for each $1\leq i\leq r$. In particular, we get that $w_{min} \in \left[1,w_{max}\right] \setminus \bigcup_{1\leq i \leq r} \left[1,w_i\right]$. 

 Now let $w'$ be any element in $\left[w_{min},w_{max}\right]\subset W^Q$. If there existed an $i$ such that $w' \in \left[1,w_i\right]$, then since $w_{min}\leq w'$, we would have $w_{min} \in \left[1,w_i\right]$, which is impossible. Hence $w' \in \left[1,w_{max}\right] \setminus \bigcup_{i=1}^r \left[1,w_i\right]$, which means that $w' \in \EE\cap W^Q$ as required.
\end{proof}

In particular, we see that parabolic orbits correspond to sub-intervals of $W^Q$. The next result describes them as the total space of a vector bundle over another generalized flag variety. 

First of all, consider the Levi decomposition $P = L \ltimes U$, where $L$ is a Levi subgroup and $U$ is the unipotent radical of $P$. If $\OO$ is a $P$-orbit associated to a double coset $W_P w_{min} W_Q$, then we define the following subset of the set $\Delta$ of simple roots of $G$ :
\[
 K_{w_{min}} = \left\{s \in \Delta(P) \mid w_{min}^{-1} s w_{min} \in \Delta(Q) \right\},
\]
where for any parabolic subgroup $R \subset G$, $\Delta(R) \subset \Delta$ is such that the associated reflections, together with $B$, generate $R$. Denote by $R_{w_{min}}$ the parabolic subgroup of $L$ generated by $K_{w_{min}}$ and $B\cap L$. We have the following geometric description of parabolic orbits :

\begin{theo}[\cite{mitchell}, Thm. 1.1]
\label{theo:para.descr}
 Let $\OO$ be the $P$-orbit associated to a double coset $W_P w_{min} W_Q$, where $P$ is a parabolic subgroup associated to a cominuscule weight. Then there exists a representation $V_{w_{min}}$ of $R_{w_{min}}$ such that $\OO \cong L \times_{R_{w_{min}}} V_{w_{min}}$ and the map $\OO \ra L/R_{w_{min}}$ is a vector bundle.
\end{theo}

\begin{rema}
 \begin{itemize}
  \item An analogous result is proved in \cite[Prop. 5]{perrin}.
  \item Note that if $P$ is not associated to a cominuscule weight, we still have a locally trivial map with affine fibers, but it is no longer a vector bundle, as stated at the end of the proof of \cite[Prop. 5]{perrin}. For instance, take $X = \Q_3 \cong\OG(1,5) \subset \P^4$ the $3$-dimensional quadric and $P=P_{\omega_2}$. The open orbit $\OO$ is $\Q_3 \setminus\P^1$. The fibration $\OO \ra \P^1$ is locally trivial, the fibers are $2$-dimensional affine spaces, but a simple calculation shows that the transition maps are quadratic.
 \end{itemize}
\end{rema}

A consequence of Thm. \ref{theo:para.descr} is that the cohomology ring of the parabolic orbit $\OO$ and of the generalized flag variety $L/R_{w_{min}}$ are isomorphic (see \cite[Chap. 3]{fulton}), which will help us to find decompositions of the Hasse diagrams in Section \ref{sec:Hasse}.

\subsection{Parabolic orbits associated to cominuscule weights in the classical Grassmannians}
\label{subsec:para.class}

For us, a \emph{classical Grassmannian} will be a homogeneous space $X=G/Q$, where $G$ is of type $A_n$, $B_n$, $C_n$ or $D_n$ and $Q$ is a maximal parabolic subgroup of $G$.
In type $A_n$, it corresponds to the usual Grassmannians $\G(m,n+1)$ for $1\leq m\leq n$, while in type $C_n$, we get the symplectic Grassmannians $\IG(m,2n)$ with $1 \leq m \leq n$. Finally, in type $B_n$ (resp. in type $D_n$), we obtain the odd orthogonal (resp. even orthogonal) Grassmannians $\OG(m,2n+1)$ (resp. $\OG(m,2n)$), where $1 \leq m \leq n$. In type $D_n$, we furthermore exclude the case where $m=n-1$, since it corresponds to a variety with Picard number two.

We start by giving the list of cominuscule weights, including the exceptional cases :

\vspace{0,3cm}
\begin{center}
\begin{tabular}{|c|l|l|}
 \hline
 Type   & Classical Grassmannians & Cominuscule weights  		    \\
 \hline
 \hline
 $A_n$  & $\G(m,n+1)$ $1 \leq m \leq n$    & $\omega_i$ $(1\leq i\leq n)$  	    \\
 \hline
 $B_n$  & $\OG(m,2n+1)$ $1 \leq m \leq n$  & $\omega_1$		       	   	    \\
 \hline
 $C_n$  & $\IG(m,2n)$ $1 \leq m \leq n$    & $\omega_n$		           	    \\
 \hline
 $D_n$  & $\OG(m,2n)$ $1 \leq m \leq n,$ $m\neq n-1$ & $\omega_1$, $\omega_{n-1}$, $\omega_n$ \\
 \hline
 $E_6$  & $E_6/P_j$ $1\leq j\leq 6$ & $\omega_1$, $\omega_6$ \\
 \hline
 $E_7$  & $E_7/P_j$ $1\leq j\leq 7$ & $\omega_7$ \\
 \hline
\end{tabular}
\end{center}

In the following sections, following Thm. \ref{theo:para.descr}, we describe the parabolic orbits associated to the above cominuscule weights for classical Grassmannians. We will not treat the exceptional cases in general since in these examples, flags and Schubert varieties are not so easily described. We will only mention the case of the Cayley plane $E_6/P_1$ in Section \ref{sec:Hasse}. However, it would probably be possible to get similar results for all exceptional cases, using the description of flags introduced by Iliev and Manivel in \cite{IM} for type $E_6$ and by Garibaldi in \cite{garibaldi} for type $E_7$. 

We will denote by $P_{\omega_i}$ the maximal parabolic subgroup containing the Borel subgroup $B$ and associated to the cominuscule fundamental weight $\omega_i$. In \ref{subsub:geom.descr}, we give a geometric description of the $P_{\omega_i}$-orbits, whereas in \ref{subsub:comb.descr}, we give a combinatorial description of the double cosets indexing them.

\subsubsection{Geometric description of parabolic orbits}
\label{subsub:geom.descr}

First we need to recall the characterization of the flag stabilized by the Borel subgroup $B$ in each of the classical types :
\begin{enumerate}[Type A$_n$ :]
 \item $B$ is the stabilizer of a (uniquely defined) complete flag 
  \[
   0 =E_0 \subset E_1 \subset \dots \subset E_n \subset E_{n+1}=\C^{n+1},
  \]
  the element $E_i$ being an $i$-dimensional subspace of $\C^{n+1}$.
 \item $B$ is the stabilizer of a \emph{type $B_n$ complete isotropic flag}
  \[
   0 =E_0 \subset E_1 \subset \dots \subset E_n \subset E_{n+1} \subset \dots \subset E_{2n} \subset E_{2n+1}=\C^{2n+1},
  \]
  where the vector spaces $E_1,\dots,E_n$ are isotropic and for each $1 \leq i \leq n$, we have $E_{n+i} = E_{n+1-i}^\perp$.
 \item $B$ is the stabilizer of a \emph{type $C_n$ complete isotropic flag}
  \[
   0 = E_0 \subset E_1 \subset \dots \subset E_n \subset E_{n+1} \subset \dots \subset E_{2n} = \C^{2n},
  \]
  where the $E_i$'s are isotropic and for each $0 \leq i \leq n$, we have $E_{n+i} = E_{n-i}^\perp$.
 \item $B$ is the stabilizer of a \emph{type $D_n$ complete isotropic flag}

  \begin{tikzpicture}[scale = 0.6]
 \node at (0,0.05) {$0$};
 \node at (1,0) {$=$};
 \node at (2,0) {$E_0$};
 \node at (3,0.05) {$\subset$};
 \node at (4,0) {$\ldots$};
 \node at (5,0.05) {$\subset$};
 \node at (6,0) {$E_{n-2}$};
 \node[rotate=35] at (7.1,0.5) {$\subset$};
 \node[rotate=-35] at (7.1,-0.5) {$\subset$};
 \node at (8,1) {$E_n$};
 \node at (8,-1) {$E'_n$};
 \node at (8,0) {$\neq$};
 \node[rotate=-35] at (8.9,0.5) {$\subset$};
 \node[rotate=35] at (8.9,-0.5) {$\subset$};
 \node at (10,0) {$E_{n+1}$};
 \node at (11,0.05) {$\subset$};
 \node at (12,0) {$\ldots$};
 \node at (13,0.05) {$\subset$};
 \node at (14,0) {$E_{2n}$};
 \node at (15,0) {$=$};
 \node at (16,0.05) {$\mathbb{C}^{2n}$};
\end{tikzpicture}

 where the vector spaces $E_1,\dots,E_{n-2}$ are isotropic, $E_n$ is a type 1 maximal isotropic subspace, $E_n'$ a type 2 isotropic subspace, $E_{n+1} = (E_n\cap E_n')^\perp$ and for each $1 \leq i \leq n-1$,  $E_{n+1+i} = E_{n-1-i}^\perp$.
\end{enumerate}

Now we prove that $P$-orbits associated to cominuscule weights in the classical Grassmannians can be described by the relative position of their elements with respect to a certain partial flag associated to the cominuscule weight defining $P$. In the following proposition, the unique complete flag stabilized by the Borel subgroup will be denoted as above .
\begin{prop}
\label{prop:geom.descr.1}
 \begin{enumerate}
  \item\label{it:A} If $X=\G(m,n+1)$ and $P=P_{\omega_i}$ for $1\leq i\leq n$, then the $P$-orbits are the
   \[
    \OO_d := \left\{ \Sigma \in X \mid \dim(\Sigma\cap E_i) = d \right\},
   \]
   for $\max(0,i+m-n-1) \leq d \leq \min(m,i)$.
  \item\label{it:B} \begin{enumerate}[a)]
         \item\label{it:B.nonmax} If $X=\OG(m,2n+1)$ with $m<n$ and $P=P_{\omega_1}$, then the $P$-orbits are
	  \begin{align*}
	    \OO_0 &:= \left\{ \Sigma \in X \mid \Sigma \not\subset E_1^{\perp} \right\}, \\
	    \OO_1 &:= \left\{ \Sigma \in X \mid \Sigma \subset E_1^{\perp} \text{ and } \Sigma \not\supset E_1 \right\}, \\
	    \OO_2 &:= \left\{ \Sigma \in X \mid \Sigma \supset E_1 \right\}.
	  \end{align*}
	 \item\label{it:B.max} If $X=\OG(n,2n+1)$ and $P=P_{\omega_1}$, then the $P$-orbits are
	  \begin{align*}
	   \OO_0 &:= \left\{ \Sigma \in X \mid \Sigma \not\supset E_1 \right\}, \\
	   \OO_1 &:= \left\{ \Sigma \in X \mid \Sigma \supset E_1 \right\}.
	  \end{align*}
        \end{enumerate}
  \item\label{it:C} If $X=\IG(m,2n)$ and $P=P_{\omega_n}$, then the $P$-orbits are the
   \[
    \OO_d := \left\{ \Sigma \in X \mid \dim (\Sigma \cap E_n) = d \right\}
   \]
   for $0 \leq d \leq m$.
  \item\label{it:D} \begin{enumerate}[a)]
         \item\label{it:D.nonmax.1} If $X=\OG(m,2n)$ with $m<n-1$ and $P=P_{\omega_1}$, then the $P$-orbits are defined as in case \ref{it:B.nonmax}.
         \item\label{it:D.nonmax.n-1} If $X=\OG(m,2n)$ with $m<n-1$ and $P=P_{\omega_{n-1}}$, then the $P$-orbits are the
	  \[
	   \OO_d := \left\{ \Sigma \in X \mid \dim (\Sigma \cap E_n') = d \right\}
	  \]
	  for $0 \leq d \leq m$.
	 \item\label{it:D.nonmax.n} If $X=\OG(m,2n)$ with $m<n-1$ and $P=P_{\omega_n}$, then the $P$-orbits are defined as in case \ref{it:D.nonmax.n-1}, with $E_n'$ replaced by $E_n$.
         \item\label{it:D.max.1} If $X=\OG(n,2n)\cong\OG(n-1,2n-1)$ and $P=P_{\omega_1}$, then the $P$-orbits are defined as in case \ref{it:B.max}.
         \item\label{it:D.max.n-1} If $X=\OG(n,2n)$ and $P=P_{\omega_{n-1}}$, then the $P$-orbits are the
	  \[
	   \OO_d := \left\{ \Sigma \in X \mid \dim (\Sigma \cap E_n') = 2d + \epsilon' \right\}
	  \]
	  for $0 \leq d \leq \lfloor \frac{n-1}{2} \rfloor$, where $\epsilon' = 0$ if $n$ is odd and $1$ if $n$ is even.
	 \item\label{it:D.max.n} If $X=\OG(n,2n)$ and $P=P_{\omega_n}$, then the $P$-orbits are defined as in case \ref{it:D.max.n-1}, with $E_n'$ replaced by $E_n$ and $\epsilon'$ replaced by $\epsilon := 1-\epsilon'$.
        \end{enumerate}
 \end{enumerate}
\end{prop}

\begin{proof}
 The parabolic subgroup $P$ is the stabilizer of the following partial flags :
 \begin{itemize}
  \item $E_i$ in case \ref{it:A} ;
  \item $E_1 \subset E_1^{\perp}$ in cases \ref{it:B}, \ref{it:D.nonmax.1} and \ref{it:D.max.1} ;
  \item $E_n=E_n^\perp$ in cases \ref{it:C}, \ref{it:D.nonmax.n} and \ref{it:D.max.n} ;
  \item $E_n'={E_n'}^\perp$ in cases \ref{it:D.nonmax.n-1} and \ref{it:D.max.n-1}.
 \end{itemize}
 Hence the dimensions of the intersections with each element of these partial flags are constant on the $P$-orbits, and conversely, the sets where these dimensions are constant are exactly the $P$-orbits.
\end{proof}

We conclude the section by giving in each classical type an explicit description of the fibration introduced in Thm. \ref{theo:para.descr}. In the following result, the orbits $\OO_d$ are the ones defined in Prop. \ref{prop:geom.descr.1}.
\begin{prop}
 \label{prop:fibr}
 \begin{enumerate}
  \item If $X=\G(m,n+1)$ and $P=P_{\omega_i}$ for $1\leq i\leq n$, then the fibrations are the
   \[
    \begin{array}{ccc}
     \OO_d & \ra & \G(d,E_i) \times G(m-d,\C^{n+1} / E_i) \\
     \Sigma & \mapsto & \left( \Sigma \cap E_i , \Sigma / (\Sigma \cap E_i) \right)
    \end{array}
   \]
  \item \begin{enumerate}[a)]
         \item If $X=\OG(m,2n+1)$ with $m<n$ and $P=P_{\omega_1}$, then the fibrations are the
	  \[
	   \begin{array}{ccc}
	    \OO_d & \ra & \OG(m-\epsilon,E_1^\perp / E_1) \\
	    \Sigma & \mapsto & \left[ \Sigma \cap E_1^\perp \right]
	   \end{array}
	  \]
	  where $\epsilon = 1$ if $d=0,2$ and $\epsilon = 0$ if $d=1$.
	 \item If $X=\OG(n,2n+1)$ and $P=P_{\omega_1}$, then the fibrations are the
	  \[
	   \begin{array}{ccc}
	    \OO_d & \ra & \OG(m-1,E_1^\perp / E_1) \\
	    \Sigma & \mapsto & \left[ \Sigma \cap E_1^\perp \right]
	   \end{array}
	  \]
        \end{enumerate}
  \item If $X=\IG(m,2n)$ and $P=P_{\omega_n}$, then the fibrations are the
   \[
    \begin{array}{ccc}
     \OO_d & \ra & \F(d,n-m+d;E_n) \\
     \Sigma & \mapsto & \left( (\Sigma \cap E_n) \subset (\Sigma^\perp \cap E_n) \right)
    \end{array}
   \]
  \item \begin{enumerate}[a)]
         \item If $X=\OG(m,2n)$ with $m<n-1$ and $P=P_{\omega_1}$, then the fibrations are defined as in case \ref{it:B.nonmax}.
         \item If $X=\OG(m,2n)$ with $m<n-1$ and $P=P_{\omega_{n-1}}$, then the fibrations are
	  \[
	   \begin{array}{ccc}
	    \OO_d & \ra & \F(d,n-m+d;E_n') \\
	    \Sigma & \mapsto & \left( (\Sigma \cap E_n') \subset (\Sigma^\perp \cap E_n') \right)
	   \end{array}
	  \]
	 \item If $X=\OG(m,2n)$ with $m<n-1$ and $P=P_{\omega_n}$, then the fibrations are defined as in case \ref{it:C}.
         \item If $X=\OG(n,2n)\cong\OG(n-1,2n-1)$ and $P=P_{\omega_1}$, then the fibrations are defined as in case \ref{it:B.max}.
         \item If $X=\OG(n,2n)$ and $P=P_{\omega_{n-1}}$, then the fibrations are
	  \[
	   \begin{array}{ccc}
	    \OO_d & \ra & \G(2d+\epsilon',E_n') \\
	    \Sigma & \mapsto & \Sigma \cap E_n' 
	   \end{array}
	  \]
	  where $\epsilon' = 0$ if $n$ is odd and $1$ if $n$ is even.
	 \item If $X=\OG(n,2n)$ and $P=P_{\omega_n}$, then the fibrations are defined as in case \ref{it:D.max.n-1}, with $E_n'$ replaced by $E_n$ and $\epsilon'$ replaced by $\epsilon := 1 - \epsilon$.
        \end{enumerate}
 \end{enumerate}
\end{prop}

\begin{proof}
 We only describe Cases \ref{it:A}, \ref{it:B.nonmax}, \ref{it:C} and \ref{it:D.max.n} with $n$ even. The other cases are very similar.

 \ref{it:A}. Since $\OO_d = \left\{ \Sigma \in X \mid \dim (\Sigma \cap E_i) = d \right\}$, the map is well defined. Moreover, the fiber at a pair $(\Sigma_1,\Sigma_2) \in \G(d,E_i) \times G(m-d,\C^{n+1} / E_i)$ is 
 \[
  \left\{ \Sigma_1 \oplus \Sigma' \mid \dim \Sigma' = m-d, \Sigma' = \Sigma_2 \text{ mod } E_i \right\} \cong \C^{\dim{\Sigma_2}\times\dim E_i} = \C^{(m-d)i}.
 \]

 \ref{it:B.nonmax}) For $d=0$, the fiber over $\Sigma_1 \in \OG(m-1,E_1^\perp/E_1)$ is
 \begin{align*}
  &\left\{ \Sigma' \oplus L \mid \Sigma' = \Sigma_1 \text{ mod } E_1, L \subset \Sigma_1^\perp \setminus E_1^\perp, L \text{ isotropic} \right\} \\ \cong & \C^{\dim \Sigma_1\times\dim E_1} \times \C^{\dim \Sigma_1^\perp -\dim \Sigma_1 -\dim L-1} = \C^{2n-m}.
 \end{align*}
 
 For $d=1$, the fiber over $\Sigma_1 \in \OG(m,E_1^\perp/E_1)$ is
 \[
  \left\{ \Sigma' \mid \Sigma' = \Sigma_1 \text{ mod } E_1 \right\} \cong \C^{\dim E_1 \dim \Sigma_1} = \C^m.
 \]

 Finally, for $d=2$, the map is an isomorphism.

 \ref{it:C}. The fiber over $(\Sigma_1 \subset \Sigma_2) \in \F(d,n-m+d;E_n)$ is
 \begin{align*}
  & \left\{ \Sigma_1 \oplus \Sigma' \mid \dim\Sigma' = m-d, \Sigma' = \Sigma_2^\perp \text{ mod } E_n, \Sigma' \subset \Sigma_1^\perp \text{ isotropic }\right\} \\
  \cong & \C^{\dim\Sigma'(\dim E_n-\dim\Sigma_1)-\frac{\dim\Sigma'(\dim{\Sigma'-1})}{2}}=\C^{(m-d)(n-d)-\frac{(m-d)(m-d-1)}{2}}.
 \end{align*}

 \ref{it:D.max.n}) We assume $n$ is even. The fiber over $\Sigma_1 \in \G(2d,E_n)$ is
 \begin{align*}
  & \left\{ \Sigma_1 \oplus \Sigma' \mid \dim \Sigma' = m-2d, \Sigma'=\Sigma_1^\perp \text{ mod } E_n \Sigma' \subset \Sigma_1^\perp, \Sigma' \text{ isotropic} \right\} \\
  \cong & \C^{\dim\Sigma'^2-\frac{\dim\Sigma'(\dim\Sigma'-1)}{2}}=\C^{(n-2d)^2-\frac{(n-2d)(n-2d-1)}{2}}. \qedhere
 \end{align*}
\end{proof}

\begin{rema}
 In Thm. \ref{theo:para.descr}, the fibrations for parabolic orbits are described combinatorially. Tedious but straightforward calculations show that these fibrations are indeed the same as those described in the above proposition.
\end{rema}

\subsubsection{Combinatorial description of parabolic orbits}
\label{subsub:comb.descr}

We begin by recalling the description of the elements of the Weyl group in type $A_n$ (respectively in types $B_n$, $C_n$ and $D_n$) as permutations (resp. signed permutations) of $\left\{ 1,\dots,n \right\}$. We do not have such a description in the exceptional cases.

In type $A$, the Weyl group is $W=\mathfrak{S}_n$, and we denote $w\in W$ as $w = (a_1, \dots, a_n)$ where $\left\{ 1,\dots,n \right\}=\left\{ a_1, \dots, a_n \right\}$, which means that $w(i)=a_i$.

In types $B_n$ and $C_n$, the Weyl group is $W=\mathfrak{S}_n\ltimes\Z_2^n$, and we denote $w\in W$ as $w = (b_1, \dots, b_n)$, where $b_i=a_i$ or $-a_i$ and $\left\{ 1,\dots,n \right\}=\left\{ a_1, \dots, a_n \right\}$, which means that $w(i)=a_i$ if $b_i=a_i$ and $w(i)=a_i$ if $b_i=-a_i$. 

Finally, in type $D_n$, the Weyl group is $W=\mathfrak{S}_n\ltimes\Z_2^{n-1}$, and we denote elements of $W$ as in the previous case, with the additional condition that the number of negative parts $-a_i$ should be even. 

We can now state a proposition describing, for all the classical types, the double coset $\EE_d\in W_P \backslash W / W_Q $ indexing the $P$-orbit $\OO_d$ defined in Prop. \ref{prop:geom.descr.1} :
\begin{prop}
 \begin{enumerate}
  \item If $X=\G(m,n+1)$ and $P=P_{\omega_i}$ for $1\leq i\leq n$, then
   \[
    \EE_d = \left\{ w \in W \mid \# \left\{ 1 \leq j \leq m \mid w(j) \leq i \right\} = m-d \right\}.
   \]
  \item \begin{enumerate}[a)]
         \item If $X=\OG(m,2n+1)$ with $m<n$ and $P=P_{\omega_1}$, then
	  \begin{align*}
	   \EE_0 &= \left\{ w \in W \mid \exists 1 \leq j \leq m, w(j)=-1 \right\} \\
	   \EE_1 &= \left\{ w \in W \mid \nexists 1 \leq j \leq m, w(j) \in \left\{ 1,-1 \right\} \right\} \\
	   \EE_2 &= \left\{ w \in W \mid \exists 1 \leq j \leq m, w(j)=1 \right\}.
	  \end{align*}
	 \item If $X=\OG(n,2n+1)$ and $P=P_{\omega_1}$, then
	  \begin{align*}
	   \EE_0 &= \left\{ w \in W \mid \exists 1 \leq j \leq m, w(j)=-1 \right\} \\
	   \EE_1 &= \left\{ w \in W \mid \exists 1 \leq j \leq m, w(j)=1 \right\}.
	  \end{align*}
        \end{enumerate}
  \item If $X=\IG(m,2n)$ and $P=P_{\omega_n}$, then
   \[
    \EE_d = \left\{ w \in W \mid \# \left\{ 1 \leq j \leq m \mid w(j) > 0 \right\} = d \right\}.
   \]
  \item \begin{enumerate}[a)]
         \item If $X=\OG(m,2n)$ with $m<n-1$ and $P=P_{\omega_1}$, then $\EE_d$ is defined as in case \ref{it:B.nonmax}.
         \item If $X=\OG(m,2n)$ with $m<n-1$ and $P=P_{\omega_{n-1}}$, then
	  \begin{align*}
	   \EE_d =& \left\{ w \in W \mid \# \left\{ j \leq m \mid w(j) > 0 \right\} = d, w(j) \neq n,-n \;\forall j\leq m \right\} \\
		 &\cup \left\{ w \mid \# \left\{ j \leq m \mid w(j) > 0 \right\} = d-1, \exists j \leq m, w(j) = -n \right\} \\
		 &\cup \left\{ w \mid \# \left\{ j \leq m \mid w(j) > 0 \right\} = d+1, \exists j \leq m, w(j) = n \right\}.
	  \end{align*}
	 \item If $X=\OG(m,2n)$ with $m<n-1$ and $P=P_{\omega_n}$, then $\EE_d$ is defined as in case \ref{it:C}.
	 \item If $X=\OG(n,2n) \cong \OG(n-1,2n-1)$ and $P=P_{\omega_1}$, then $\EE_d$ is defined as in case \ref{it:B.max}.
	 \item If $X=\OG(n,2n)$ and $P=P_{\omega_{n-1}}$, then
	  \begin{align*}
	   \EE_d = &\left\{ w \in W \mid \# \left\{ j \mid w(j) > 0 \right\} = 2d+\epsilon'-1 \text{ and } \exists j, w(j) = -n \right\} \\
		 &\cup \left\{ w \in W \mid \# \left\{ w(j) > 0 \right\} = 2d+\epsilon'+1 \text{ and } \exists j , w(j) = n \right\},
	  \end{align*}
	  where $\epsilon'=0$ if $n$ is odd and $1$ if $n$ is even.
	 \item If $X=\OG(n,2n)$ and $P=P_{\omega_n}$, then
	  \[
	   \EE_d = \left\{ w \in W \mid \# \left\{ j \mid w(j) > 0 \right\} = 2d + \epsilon \right\},
	  \]
	  where $\epsilon=1-\epsilon'$.
        \end{enumerate}
 \end{enumerate}
\end{prop}

\begin{proof}
 The arguments for each case being similar, we only prove the proposition in Case \ref{it:D.nonmax.n-1}, which is a little more complicated than the others. 

 Here the Weyl groups are $W=\mathfrak{S}_n \ltimes \Z_2^{n-1}$, $W_P=\mathfrak{S}_{n-1}\ltimes\Z_2$ and $W_Q=\mathfrak{S}_m\times(\mathfrak{S}_{n-m}\ltimes\Z_2^{n-m-1})$. We will denote elements of $W$ as signed permutations $w=(b_1,\dots,b_n)$ as in the beginning of the section. 

 The action of $W_Q$ on the right permutes the $m$ first entries $b_1,\dots,b_m$ of $w$ on one hand, and the $n-m$ last entries $b_{m+1},\dots,b_n$ on the other hand, and changes the sign of these last entries while keeping the total number of minus signs even. Hence the minimal length representatives of classes in $W/W_Q$ are of the form :
 \[
  w = \left( u_1<\dots<u_l,-z_{m-l}<\dots<-z_1,v_1<\dots<v_{n-m-1},(-1)^{m-l} v_{n-m} \right),
 \]
 where $0 \leq l \leq m$, $\{ u_i \} \cup \{ z_r \} \cup \{ v_j \} = \{1,\dots,n\}$ and $v_{n-m-1} < v_{n-m}$.

 Moreover, the action of $W_P$ on the right permutes the $n-1$ values $1, \dots, n-1$ and exchanges $n-1$ and $n$ while changing their signs. Hence the minimal length representatives of double cosets in $W_P \backslash W/W_Q$ are of the form :
 \[
  w_0 = id \text{ or } w_d = \left( 1 < \dots < d-1 < n , -n+1 < \dots < -n+m-d , \underline{v} \right),
 \]
 where $1\leq d \leq m$ and
 \[
  \underline{v} = \left(d < \dots < n-m+d-2 , (-1)^{m-d} (n-m+d-1)\right).
 \]
 Now it is enough to prove that all elements of the set $\EE_d$ defined in the statement of the proposition are in the same double coset as $w_d$.

 First suppose $w\in W$ is such that $\# \left\{ j \leq m \mid w(j) < 0 \right\} = d$ and $w(j) \neq n,-n$ for all $j\leq m$. Using the action of $W_Q$ on the right, we see that $w$ is in the same double coset as
 \[
  w^1 = \left( a_1 < \dots < a_d, -b_{m-d} < \dots < -b_1, c_1 < \dots < c_{n-m-1}, (-1)^{m-d} n \right).
 \]
 Using (several times) the action of the simple reflections $s_1,\dots,s_{n-1}$ of $W_P$ on the left (which together permute the values from $1$ to $n-1$), we deduce that $w^1$ is in the same double coset as
 \[
  w^2 = \left( 1 < \dots < d, -n+1 < \dots < -n+m-d, \underline{v} \right),
 \]
 where $\underline{v} = \left( d+1 < \dots < n-m+d-1, (-1)^{m-d} n \right)$.
 Then applying the simple reflection $s_n \in W_P$ on the left, we get
 \[
  w^3 = \left( 1 < \dots < d < n, -n+2 < \dots < -n+m-d, \underline{v} \right),
 \]
 where $\underline{v} = \left( d+1 < \dots < n-m+d-1, (-1)^{m-d+1} (n-1) \right)$.

 Finally, using the action of the simple reflections $s_1,\dots,s_{n-1}$ of $W_P$ on the left, we obtain the element $w^4=w_d$, which proves that $w$ is in the same double coset as $w_d$. 

 The reasoning in the two other situations ($\# \left\{ j \leq m \mid w(j) > 0 \right\} = d-1$ and $\exists j \leq m, w(j) = -n$ on one hand, $\# \left\{ j \leq m \mid w(j) > 0 \right\} = d+1$ and $\exists j \leq m, w(j) = n$ on the other hand) being very similar, this concludes the proof.
\end{proof}

\begin{defi}
\label{def:degree}
 Let $w\in W$ be an element of the Weyl group. Then $w$ belongs to one of the double cosets $\EE_d$ defined in the statement of the proposition and we define the integer $d(w):=d$.
\end{defi}

\section{Link between \texorpdfstring{$P$}{P}-orbits and quantum product}
\label{sec:link.para}

Here we describe the link between Thm. \ref{th:para.formula} and parabolic orbits for homogeneous spaces $X=G/Q$, where $Q$ is a maximal parabolic subgroup. Since $Q$ is maximal, we have $\Phi^\vee / \Phi^\vee_Q \cong \Z$. Hence for each $w \in W$, we may define an integer
\[
 \delta(w) := \eta_Q(\omega_i^\vee - w^{-1}(\omega_i^\vee)).
\]
In the following sections, we will prove that the loci where $\delta(w)$ is constant correspond to the double cosets $\EE$ indexing $P$-orbits. For classical Grassmannians, this proves that for every $w \in W ^Q$, $\delta(w)$ equals the integer $d(w)$ introduced in Definition \ref{def:degree}.

\subsection{The integer \texorpdfstring{$\delta(w)$}{delta(w)} is constant on parabolic orbits.}

We start by proving that $\delta$ is constant on the double cosets $\EE = W_P w W_Q$.
Consider $w' \in \EE$. From the definition of $\EE$, it follows that $w'$ can be written as $w_P w w_Q$ for some $w_P \in W_P$ and $w_Q\in W_Q$. Denote by $\omega_i$ the cominuscule weight defining $P$. Reflections associated to the simple roots will be denoted by $s_l$ for $1 \leq l \leq n$. 

If $l \neq i$, we have 
\[
 s_l(\omega_i^\vee)=\omega_i^\vee-(\alpha_l,\omega_i^\vee)\alpha_l^\vee=\omega_i^\vee,
\]
hence
\begin{equation}
\label{eq:P-orb.1}
 w_P^{-1}(\omega_i^\vee)=\omega_i^\vee.
\end{equation}
Now consider $e:= \eta_Q(w^{-1}(\omega_i^\vee))$. Then by definition of $\eta_Q$,
\[
 w^{-1}(\omega_i^\vee)=e\alpha_m^\vee+\sum_{p\neq m}c_p\alpha_p^\vee,
\]
where the $c_p$ are some coefficients. But if $l\neq m$, we have
\[
 s_l(\alpha_m^\vee)=\alpha_j^\vee-(\alpha_l,\alpha_m^\vee)\alpha_l^\vee.
\]
Similarly, for $p\neq m$ and $l\neq p,m$ :
\[
 s_l(\alpha_p^\vee)=\alpha_p^\vee-(\alpha_l,\alpha_p^\vee)\alpha_l^\vee,
\]
and if $p\neq m$ and $l=p$ :
\[
 s_p(\alpha_p^\vee)=-\alpha^\vee_p.
\]

Hence if we apply the reflection $s_l$ for $l\neq m$, the coefficient of $\alpha_m^\vee$ does not change.
We conclude that $\eta_Q\left(w_Q^{-1}w^{-1}\omega_i^\vee\right)=\eta_Q\left(w^{-1}\omega_i^\vee\right)$. Using Equation \eqref{eq:P-orb.1}, we obtain
\[
 \eta_Q(w_Q^{-1}w^{-1}w_P^{-1}\omega_i^\vee)=\eta_Q(w^{-1}\omega_i^\vee).
\]

\subsection{The integer \texorpdfstring{$\delta(w)$}{delta(w)} changes on different parabolic orbits.}

It is enough to prove that if $w' \in \EE' \cap W^Q$ is a successor of $w \in \EE \cap W^Q$ for the Bruhat order in $W^Q$, where $\EE$ and $\EE'$ are two different $P$-orbits, then $\delta(w') > \delta(w)$.

Since $w$ and $w$ and $w'$ do not belong to the same $P$-orbit, we know that $w'=s_{\alpha_0} w$ for some positive root $\alpha_0\in \Phi^+\setminus \left( \Phi_P^+ \cap \Phi_Q^+ \right\}$. Indeed, if $\alpha\in \Phi_P^+$, then the reflection $s_\alpha$ is in $W_P$, hence stabilizes $\EE$ and if $\alpha \in \Phi_Q^+$, then $w' = w$ in $W/W_Q$. Moreover, we have $l_Q(w')=l_Q(w)+1$, where $l_Q$ is the length function of $W^Q$.

We set $L_Q(w) := \left\{ \alpha \in \Phi^+ \setminus \Phi^+_Q \mid w(\alpha) \in \Phi^- \right\}$.
There exists $\beta_0\in \Phi^+ \setminus \Phi^+_Q$ such that $w(\beta_0)=\alpha_0$. Indeed, if it were not the case, then for all $\alpha\in L_Q(w')$, we would have $s_{\alpha_0}w(\alpha)\in \Phi^-$ and $w(\alpha)\neq \alpha_0$, hence $w(\alpha)\in \Phi^-$ and $\alpha\in L_Q(w)$. This would mean that $l_Q(w')\leq l_Q(w)$, which is absurd. 

Let us now compute $\delta(w')$ :
\[
 \delta(w')=\eta_Q\left(\omega_i^\vee-w^{-1}s_{\alpha_0}(\omega_i^\vee)\right)=\delta(w)+(\alpha_0,\omega_i^\vee)\eta_Q\left(w^{-1}\alpha^\vee_0\right).
\]
Since $\alpha_0\in \Phi^+\setminus \Phi_P^+$, we have $(\alpha_0,\omega_i^\vee)>0$. Moreover, $w(\beta_0)=\alpha_0$ implies $w^{-1}(\alpha_0^\vee)=\beta_0^\vee$, and $\eta_Q(\beta_0)>0$ since $\beta_0\in \Phi^+\setminus \Phi_Q^+$. Finally $\delta(w')>\delta(w)$ as required. 

We conclude that the loci
\[
 \left\{ w \in W^Q \mid \delta(w) = d \right\}
\]
coincide with the sets $\EE \cap W^Q$.

\section{Decomposition of the Hasse diagram}
\label{sec:Hasse}

In \cite{CMP2}, Chaput, Manivel and Perrin relate the quantum product by the point class in minuscule varieties with a decomposition of their Hasse diagram. The \emph{Hasse diagram $\HH$} of a homogeneous space with Picard rank one is the diagram of the multiplication by the hyperplane class $h$. More precisely, its vertices are the Schubert classes $\sigma_w$ for $w\in W^Q$ and $\sigma_v$ and $\sigma_w$ are related by an arrow of multiplicity $r$ if and only if $\sigma_w$ appears with multiplicity $r$ in the cup-product $\sigma_v \cup h$.

The results of previous sections enable us to find decompositions of the Hasse diagram in the non-minuscule case, corresponding to the quantum product by the Schubert classes $\sigma_{v_i}$ associated to cominuscule weights introduced in the statement of Thm. \ref{th:para.formula}. 

Let $\OO$ be a $P$-orbit of $X$. It is the union of the Schubert cells $C_w\subset X$ for all $w$ in the associated double coset $\EE$.  The set $\EE\cap W^Q$ being an interval (cf Prop. \ref{prop:interval}), we denote it as $\EE\cap W^Q=\left[w_{min},w_{max}\right]$. From Thm. \ref{th:para.formula}, we know that $\OO$ is a vector bundle over the generalized flag variety $F:=L/R_{w_{min}}$.

Here we state a result relating the Hasse diagrams of the parabolic orbit $\OO$ with a similar diagram for the flag variety $F$ :
\begin{prop}
\label{prop:link.Hasse}
 Let $\psi : \OO \ra F$ be the fibration, $i : \OO \hookrightarrow X$ the natural embedding and $h$ the hyperplane class of $X$. Then :
 \begin{enumerate}
  \item There exists a class $h'\in \H^2(F)$ such that $i^* h =\psi^* h'$ ;
  \item The Hasse diagram of $\OO$ is isomorphic to the diagram of the multiplication by $h'$ in $F$.
 \end{enumerate}
\end{prop}

\begin{proof}
 \begin{enumerate}
  \item Since $i^* h\in\H^2(\OO)\cong\H^2(F)$, there exists $h'\in\H^2(F)$ such that $i^* h  =\psi^* h'$.
  \item There exists an isomorphism $W^F \cong \EE\cap W^Q$, where $W^F$ is the set of minimal length representatives of $W_L/W_{R_{w_{min}}}$. 
   Indeed, let $C_u^F$ be a Schubert cell of $F$. Since $\psi$ is a vector bundle, its inverse image $\psi^{-1}(C_u^F)$ is a Schubert cell of $X$, which we denote by $C_{\phi(u)}^X$, where $\phi(u) \in W^Q$. Since $C_{\phi(u)}^X \subset \OO$, we have $\phi(u) \in \EE \cap W^Q$, and $\phi$ is the desired isomorphism. It yields a correspondence between the vertices of the Hasse diagram of $\OO$ and those of the diagram of the multiplication by the class $h'$ in $F$. 
 
  Now we study the correspondence between the edges of both diagrams. Assume that 
  \[
   [Y_w] \cup h' = \sum_v a_v [Y_v],
  \]
  where $Y_v$ denotes the Schubert variety of $F$ associated to the element $v$. This means that a generic hyperplane section of $Y_w$ is rationally equivalent to the union of the $Y_v$ with multiplicities $a_v$. Let $Y_u$ be a Schubert variety of $F$. Its inverse image $\psi^{-1}(Y_u)$ is the closure in $\OO$ of the Schubert cell $C_{\phi(u)}^X$, hence it is the intersection of $\OO$ with the Schubert variety $X_{\phi(u)}$.
  
  Thus $X_{\phi(w)}\cap \OO$ is rationally equivalent to the union of the $X_{\phi(v)}\cap \OO$ with multiplicities $a_v$. As a consequence, if $H$ is a generic hyperplane, a section $X_{\phi(w)}\cap\OO\cap H$ is rationally equivalent to the union of the $X_{\phi(v)}\cap\OO\cap H$ with multiplicities $a_v$. If we consider the closure in $\overline{\OO}$, we deduce that $X_{\phi(w)}\cap H$ is rationally equivalent to the sum of the $X_{\phi(v)}$ with multiplicities $a_v$, plus a class $Z$ supported in the boundary $\overline{\OO}\setminus \OO$. But such a class is rationally equivalent to the union of some Schubert varieties $X_u$ contained in $\overline{\OO}\setminus \OO$, with some multiplicities $b_u$. This rational equivalence stays true in the whole of $X=G/P_J$. Taking cohomology classes, it means that 
  \[
   \sigma_{\phi(w)} \cup h = \sum_v a_v \sigma_{\phi(v)} + \sum_u b_u \sigma_u. 
  \]
  Since the Schubert varieties $X_u$ are contained in $\overline{\OO} \setminus \OO$, the elements $u \in W^Q$ are not contained $\EE \cap W^Q$. Hence they do not contribute to the arrows of the Hasse diagram of $\OO$. This proves that the Hasse diagram of $\OO$ has the same arrows as the diagram of the multiplication by the class $h'$ in $F$. \qedhere
 \end{enumerate}
\end{proof}

We may now conclude by combining the previous results to describe the Hasse diagrams of the classical Grassmannians :

\begin{theo}
\label{theo:decomp}
\begin{enumerate}
 \item In types $A_n$, $C_n$, $D_n$, and in type $B_n$ for odd orthogonal Grassmannians $\OG(m,2n+1)$ with $m\neq n-1$, if $\OO$ is a parabolic orbit associated to a cominuscule weight $\omega_i$, the Hasse diagrams $\HH_\OO$ and $\HH_F$ of $\OO$ and the corresponding flag variety $F$ described in Prop. \ref{prop:geom.descr.1} are isomorphic.
 \item In type $B_n$ for the odd orthogonal Grassmannian $\OG(n-1,2n+1)$, if we denote by $\OO_0$, $\OO_1$ and $\OO_2$ the parabolic orbits associated to the weight $\omega_1$ and $F_0$, $F_1$ and $F_2$ the corresponding flag varieties, we have $\HH_{\OO_0} \cong \HH_{F_0}$ and $\HH_{\OO_2} \cong \HH_{F_2}$, but $\HH_{\OO_1}$ corresponds to $\HH_{F_1}$ with the multiplicities of the arrows doubled.
\end{enumerate}
\end{theo}

\begin{proof}
 Since we want to apply Prop. \ref{prop:link.Hasse}, it is enough to compute the class $h' \in \H^2(F)$ introduced in the statement of this proposition. We use the same notations.

 In type $A_n$, denote by $\SS$ the tautological bundle on $X$ and $\SS_1,\SS_2$ the tautological bundles on $F$. We need to prove that $i^*(\det\SS)=\psi^*(\det\SS_1\otimes\det\SS_2)$, which is simply the consequence of the exact sequence
 \[
  0 \ra \psi^*\SS_1 \ra i^* \SS \ra \psi^*\SS_2 \ra 0
 \]
 since $h=c_1(\det\SS)$ and $h'=c_1(\det\SS_1\otimes\det\SS_2)$.

 In type $B_n$ for $X=\OG(m,2n+1)$ with $m<n$, we will prove for each of the three $P$-orbits $\OO_d$ for $d=0,1,2$ that $i^*(\det\SS)=\psi^*(\det\SS_1)$, where $\SS_1$ is the tautological bundle on $F$. Indeed, for $d=0$, we have the exact sequences
 \begin{align*}
  0 &\ra \Sigma \cap E_1^\perp \ra \Sigma \ra \Sigma / (\Sigma \cap E_1^\perp) \ra 0 \\
  0 &\ra \Sigma' \ra E_1^\perp / E_1 \ra E_1^\perp / (\Sigma\cap E_1^\perp \oplus E_1) \ra 0 \\
  0 &\ra \Sigma \cap E_1^\perp \ra \Sigma\cap E_1^\perp \oplus E_1 \ra (\Sigma\cap E_1^\perp \oplus E_1) / (\Sigma \cap E_1^\perp) \ra 0,
 \end{align*}
 which give the following equalities of determinant bundles
 \begin{align*}
  \det(\Sigma) &= \det(\Sigma\cap E_1^\perp)\otimes\det(\Sigma / (\Sigma \cap E_1^\perp)) \\
  \det(\Sigma') &= \det(\Sigma\cap E_1^\perp \oplus E_1) \\
  \det(\Sigma\cap E_1^\perp \oplus E_1) &= \det(\Sigma\cap E_1^\perp)\otimes\det((\Sigma\cap E_1^\perp \oplus E_1) / (\Sigma \cap E_1^\perp)).
 \end{align*}
 We conclude by using the fact that the quadratic form induces a duality
 \[
  \Sigma / (\Sigma\cap E_1^\perp) \times (\Sigma\cap E_1^\perp \oplus E_1) / (\Sigma \cap E_1^\perp) \ra \C.
 \]
 For $d=1$, we use the same method, only replacing $\Sigma\cap E_1^\perp$ with $\Sigma$, and for $d=2$, the result follows from the exact sequence
 \[
  0 \ra E_1 \ra \Sigma \ra \Sigma / E_1 \ra 0.
 \]
 Now we have proved that $i^*(\det\SS)=\psi^*(\det\SS_1)$, it remains to relate their first Chern classes with the classes $h$ and $h'$ defined in the statement of Prop. \ref{prop:link.Hasse}. We always have $c_1(\det\SS)=h$, but there are two cases for $c_1(\det\SS_1)$ :
 \begin{align*}
  c_1(\det\SS_1) = \begin{cases}
                   h' &\text{if $m<n-1$} \\
                   2h' &\text{if $m=n-1$.}
                  \end{cases}
 \end{align*}
 Indeed, $\OG(n-1,2n-1)$ is projectively isomorphic to $\OG(n-1,2n-2)$, which is embedded in $\P(V_{\omega_{n-1}})$, where $V_{\omega_{n-1}}$ is the half-spin representation. Hence the hyperplane class $h'$ is equal to the first Chern class of the line bundle associated to the weight $\omega_{n-1}$, while $\det\SS_1$ is the line bundle associated to the weight $2\omega_{n-1}$.

 In type $B_n$ for $X=\OG(n,2n+1)$, we prove as in the non-maximal case that $i^*(\det\SS)=\psi^*(\det\SS_1)$, $c_1(\det\SS)=h$ and $c_1(\det\SS_1)=h'.$

 In type $C_n$, denote by $\SS$ the tautological bundle on $X$ and $\SS_1,\SS_2$ the tautological bundles on $F$. Since $h=c_1(\det\SS)$ and $h'=c_1(\det\SS_1\otimes\det\SS_2)$, we need to prove that $i^*(\det\SS)=\psi^*(\det\SS_1\otimes\det\SS_2)$, which is simply the consequence of the exact sequences
 \[
  0 \ra \Sigma\cap E_n \ra \Sigma \ra \Sigma/(\Sigma\cap E_n) \ra 0
 \]
 \[
  0 \ra \Sigma^\perp\cap E_n \ra E_n \ra E_n/(\Sigma^\perp\cap E_n) \ra 0.
 \]

 In type $D_n$ for $X=\OG(m,2n)$ with $m<n$ or for $X=\OG(n,2n)$ with $P=P_{\omega_1}$, the result is proven in an analogous way as in types $B_n$ and $C_n$. This leaves us with the case where $X=\OG(n,2n)$ and $P=P_{\omega_n}$ or $P_{\omega_{n-1}}$. Here we treat the case $P=P_{\omega_n}$, the other being very similar. We use the two exact sequences
 \[
  0 \ra \Sigma\cap E_n \ra \Sigma \ra \Sigma / (\Sigma \cap E_n) \ra 0
 \]
 \[
  0 \ra \Sigma\cap E_n \ra E_n \ra E_n / (\Sigma \cap E_n) \ra 0
 \]
 and the duality $\Sigma / (\Sigma \cap E_n) \times E_n / (\Sigma \cap E_n) \ra \C$ to prove that $i^*(\det\SS)=\psi^*(\det\SS_1)$, with notations as before. Then we use the fact that $h=c_1(\det\SS)$ and $h'=c_1(\det\SS_1)$. 
\end{proof}

Finally, we give some pictures illustrating Thm. \ref{theo:decomp}. We start with a type $C_n$ example : the symplectic Grassmannian $\IG(2,8)$ in Figure \ref{IG(2,8)}. There are three orbits, two being vector bundles over the Grassmannian $\G(2,4)$ and another over the two-step flag variety $\F(1,3;4)$.
\begin{figure}[ht]
 \centering
 \caption{$P_{\omega_4}$-orbits in $\IG(2,8)$}
 \begin{tikzpicture}[scale=2/3]
 
\tikzstyle{every node}=[draw,circle,minimum size=6pt,inner sep=0pt]

\draw[blue,thick] (0,0) node[fill=blue] (1) {}
	-- (1,0) node[fill=blue] (2) {}
	-- (2,1) node[fill=blue] (3) {}
	-- (3,-1) node[fill=blue] (5) {}
	-- (4,-2) node[fill=blue] (6) {}
		  (2)
	-- (2,-1) node[fill=blue] (4) {}
	-- (5);
\node[draw=none,text=blue,thick] at (0.8,1.1) {$\mathrm{G}(2,4)$};

\draw[red,thick] (3,1) node[fill=red] (7) {}
	-- (4,2) node[fill=red] (8) {}
	-- (5,2) node[fill=red] (10) {}
	-- (6,0) node[fill=red] (14) {}
	-- (7,2) node[fill=red] (16) {}
	-- (8,1) node[fill=red] (18) {}
	-- (7,0) node[fill=red] (17) {}
	-- (14)
	-- (5,-2) node[fill=red] (12) {}
	-- (4,0) node[fill=red] (9) {}
	-- (7)
		 (8)
	-- (5,0) node[fill=red] (11) {}
	-- (6,2) node[fill=red] (13) {}
	-- (16)
		 (9)
	-- (11)
	-- (6,-2) node[fill=red] (15) {}
	-- (17);
\draw[red,thick,double] (10) -- (13)
			(11) -- (14)
			(12) -- (15);
\node[draw=none,text=red,thick] at (5.5,3) {$\mathrm{F}(1,3;4)$};

\draw[green,thick] (7,-2) node[fill=green] (19) {}
	-- (8,-1) node[fill=green] (20) {}
	-- (9,1) node[fill=green] (21) {}
	-- (10,0) node[fill=green] (23) {}
	-- (11,0) node[fill=green] (24) {}
		   (20)
	-- (9,-1) node[fill=green] (22) {}
	-- (23);
\node[draw=none,text=green,thick] at (10.2,1.1) {$\mathrm{G}(2,4)$};

\draw (3) -- (7)
      (5) -- (9)
      (6) -- (12)
      (15) -- (19)
      (17) -- (20)
      (18) -- (21);

\end{tikzpicture}
 \label{IG(2,8)}
\end{figure}
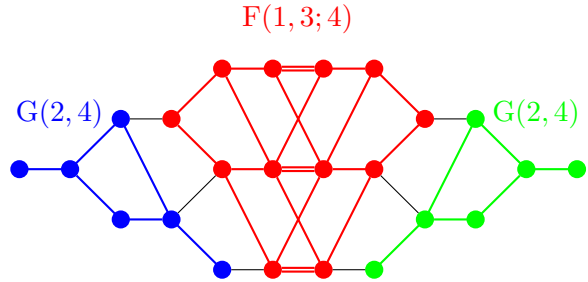

Then we consider a type $B_n$ example : the odd orthogonal Grassmannian $\OG(3,9)$ in Figure \ref{OG(3,9)}. There are again three orbits. The first and last are vector bundles over $\OG(2,7)$. For the middle orbit, which is a vector bundle over $\OG(3,7)$, we see as expected that the multiplicity of all arrows is multiplied by $2$.
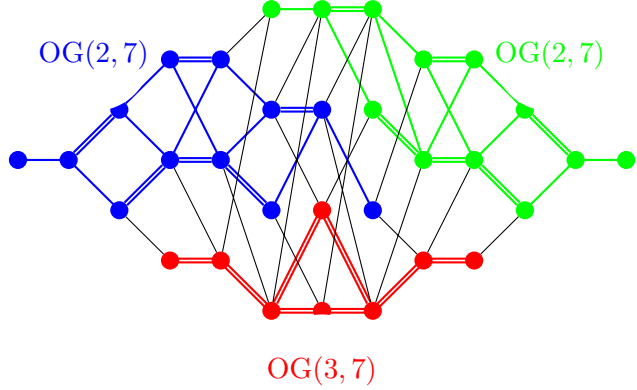
\begin{figure}
 \centering
 \caption{$P_{\omega_{1}}$-orbits in $\OG(3,9)$}
 \begin{tikzpicture}[scale=2/3]
 
\tikzstyle{every node}=[draw,circle,fill=white,minimum size=6pt,inner sep=0pt]

\draw[blue,thick] (0,0) node[fill=blue] (1) {}
        -- (1,0) node[fill=blue] (2) {}
        -- (2,-1) node[fill=blue] (4) {}
		  (2,1) node[fill=blue] (3) {}
	-- (3,0) node[fill=blue] (6) {}
	-- (4,2) node[fill=blue] (7) {}
	-- (5,1) node[fill=blue] (9) {}
	-- (4,0) node[fill=blue] (8) {}
	-- (3,2) node[fill=blue] (5) {}
	-- (3)
		  (5,-1) node[fill=blue] (10) {}
	-- (6,1) node[fill=blue] (11) {}
	-- (7,-1) node[fill=blue] (12) {};
\draw[blue,thick,double] (2) -- (3)
			 (4) -- (6)
			 (5) -- (7)
			 (6) -- (8)
			 (8) -- (10)
			 (9) -- (11);
\node[draw=none,text=blue,thick] at (1.5,2.1) {$\mathrm{OG}(2,7)$};

\draw[red,double,thick] (3,-2) node[fill=red] (13) {}
	-- (4,-2) node[fill=red] (14) {}
	-- (5,-3) node[fill=red] (15) {}
	-- (6,-3) node[fill=red] (17) {}
	-- (7,-3) node[fill=red] (18) {}
	-- (8,-2) node[fill=red] (19) {}
	-- (9,-2) node[fill=red] (20) {}
		  (15)
	-- (6,-1) node[fill=red] (16) {}
	-- (18);
\node[draw=none,text=red,thick] at (6,-4.2) {$\mathrm{OG}(3,7)$};

\draw[green,thick] (5,3) node[fill=green] (21) {}
	-- (6,3) node[fill=green] (22) {}
	-- (7,1) node[fill=green] (24) {}
		   (7,3) node[fill=green] (23) {}
	-- (8,0) node[fill=green] (26) {}
	-- (9,2) node[fill=green] (27) {}
	-- (10,1) node[fill=green] (29) {}
	-- (9,0) node[fill=green] (28) {}
	-- (8,2) node[fill=green] (25) {}
	-- (23)
		   (10,-1) node[fill=green] (30) {}
	-- (11,0) node[fill=green] (31) {}
	-- (12,0) node[fill=green] (32) {};
\draw[green,thick,double] (22) -- (23)
			  (24) -- (26)
			  (25) -- (27)
			  (26) -- (28)
			  (28) -- (30)
			  (29) -- (31);
\node[draw=none,text=green,thick] at (10.5,2.1) {$\mathrm{OG}(2,7)$};

\draw (4) -- (13)
	     (6) -- (14)
	     (8) -- (15)
	     (9) -- (16)
	     (11) -- (18)
	     (12) -- (19)
	     (10) -- (17)
	     (14) -- (21)
	     (15) -- (22)
	     (17) -- (23)
	     (16) -- (24)
	     (18) -- (26)
	     (19) -- (28)
	     (20) -- (30);
\draw (7) -- (21)
		    (9) -- (22)
		    (11) -- (23)
		    (12) -- (25);

\end{tikzpicture}
 \label{OG(3,9)}
\end{figure}

Finally, let us recall an exceptional example, computed in \cite{CMP2} : the Cayley plane $X=E_6/P_{\omega_1}=\mathbb{OP}^2$ (see Figure \ref{E6-P1}). There are three $P_{\omega_1}$ orbits. Indeed, we know that a partial $E_6$-flag associated to $P_{\omega_1}$ simply consists in a point $p_0 \in X$. The $P_{\omega_1}$-orbits are 
\begin{align*}
 \OO_0 &= \left\{ p \in X \mid p \not\in \text{ line through } p_0 \right\} \\
 \OO_1 &= \left\{ p \in X \mid p \in \text{ line through } p_0, p\neq p_0 \right\} \\
 \OO_2 &= \left\{ p_0 \right\}.
\end{align*}
We can also describe these orbits as vector bundles over generalized flag varieties
\begin{align*}
 \OO_0 &\ra \Q_8 \\
 \OO_1 &\ra \mathbb{S}_{10} \\
 \OO_2 &\ra \text{pt},
\end{align*}
where $\Q_8\cong\mathbb{OP}^1$ is the $8$-dimensional quadric and $\mathbb{S}_{10} \cong \OG(5,10)$ is the $10$-dimensional spinor variety. Indeed, the last fibration is trivial and the second stems from the description of $\OO_1$ as a cone over $\mathbb{S}_{10}$ (see \cite[Lemma 4.1]{IM}). Finally, we know from \cite{IM} that the Cayley plane also parametrises the family of $\Q_8$'s it contains, hence to $p_0$ is associated an $8$-dimensional quadric $Q_0$. The same goes for $p$, to which corresponds a quadric $Q$. These quadrics are isomorphic to projective octonionic lines $\mathbb{OP}^1$, and two general such lines meet in one point in $\mathbb{OP}^2$, hence the first fibration.

\begin{figure}
 \centering
 \caption{$P_{\omega_{6}}$-orbits in $E_6/P_{\omega_1}$}
  \begin{tikzpicture}[scale=2/3]
 
\tikzstyle{every node}=[draw,circle,fill=white,minimum size=6pt,inner sep=0pt]

\draw[blue,thick] (0,0) node[fill=blue] (0) {}
        -- (1,0) node[fill=blue] (1) {}
        -- (2,0) node[fill=blue] (2) {}
        -- (3,0) node[fill=blue] (3) {}
        -- (4,1) node[fill=blue] (4) {}
        -- (5,0) node[fill=blue] (6) {}
        -- (6,1) node[fill=blue] (7) {}
        -- (7,2) node[fill=blue] (8) {}
	-- (8,3) node[fill=blue] (9) {};
\draw[blue,thick] (3) -- (4,-1) node[fill=blue] (5) {}
	    -- (6);
\node[draw=none,text=blue,thick] at (2,1.5) {$\mathbb{Q}_8$};

\draw[red,thick] (5,-2) node[fill=red] (10) {}
        -- (6,-1) node[fill=red] (11) {}
        -- (7,0) node[fill=red] (12) {}
        -- (8,1) node[fill=red] (13) {}
        -- (9,2) node[fill=red] (15) {}
        -- (10,1) node[fill=red] (17) {}
        -- (11,0) node[fill=red] (19) {}
        -- (12,1) node[fill=red] (21) {}
	-- (13,0) node[fill=red] (23) {}
	-- (14,0) node[fill=red] (24) {}
	-- (15,0) node[fill=red] (25) {};
\draw[red,thick] (12) -- (8,-1) node[fill=red] (14) {}
	-- (9,0) node[fill=red] (16) {}
	-- (10,-1) node[fill=red] (18) {}
	-- (11,-2) node[fill=red] (20) {}
	-- (12,-1) node[fill=red] (22) {}
	-- (23);
\draw[red,thick] (13) -- (16)
	-- (17)
	   (18) -- (19)
	-- (22)
	-- (23);
\node[draw=none,text=red,thick] at (7,-2) {$\mathrm{OG}(5,10)$};

\draw[green,thick] (16,0) node[fill=green] (26) {};
\node[draw=none,text=green,thick] at (16,1) {pt};

\draw (5) -- (10)
      (6) -- (11)
      (7) -- (12)
      (8) -- (13)
      (9) -- (15)
      (25) -- (26);

\end{tikzpicture}
 \label{E6-P1}
\end{figure}
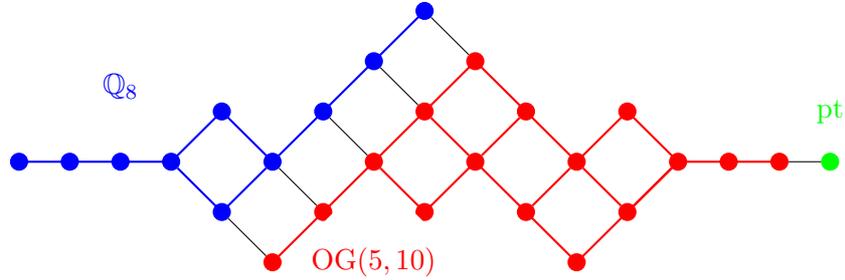


\bibliographystyle{alpha}
\bibliography{biblio}

\end{document}